\documentclass[10pt]{amsart}
\usepackage{enumerate,amsmath,amssymb,latexsym,
amsfonts, amsthm, amscd, MnSymbol}

\theoremstyle{plain}

\newtheorem{theorem}{Theorem}

\numberwithin{equation}{section}

\newcommand{\R}{\mathbb{R}}

\newcommand{\kk}{\kappa}

\newcommand{\rms}{{\rm s}}
\newcommand{\rmc}{{\rm c}}

\begin{document}

\title {Hypergeometric identities in elliptic signature six}

\date{}

\author[P.L. Robinson]{P.L. Robinson}

\address{Department of Mathematics \\ University of Florida \\ Gainesville FL 32611  USA }

\email[]{paulr@ufl.edu}

 \keywords{}

\begin{abstract}

Within the Ramanujan theories of elliptic functions, Li-Chien Shen constructed natural elliptic functions in signature three and signature four. When applied in signature six, the same constructions produce non-elliptic functions that nevertheless engender the corresponding hypergeometric identities of Ramanujan. 

\end{abstract}

\maketitle

\medbreak

\section*{Introduction}

\medbreak 

In his second notebook, Ramanujan laid the foundations for his theory of elliptic functions to alternative bases, in which the hypergeometric function $F(\tfrac{1}{2}, \tfrac{1}{2}; 1 ; \bullet)$ familiar from classical elliptic function theory is replaced by one of three others, in what are now known as the theories of signature three, signature four, and signature six. See [1] for a substantial account of these developments. 

\medbreak 

The signature three theory involves $F(\tfrac{1}{3}, \tfrac{2}{3}; 1 ; \bullet)$. In this case, Shen [9] showed the presence of a certain elliptic function ${\rm dn}_3$ within the theory;  as one of several byproducts, his investigations yielded a fresh derivation of the Borwein cubic identity. Later, the elliptic function ${\rm dn}_3$ was somewhat differently derived in [4] and considered at greater depth in [7]. 

\medbreak 

The signature four theory involves $F(\tfrac{1}{4}, \tfrac{3}{4}; 1 ; \bullet)$. In this case, Shen [11] again showed the presence of an elliptic function ${\rm dn}_2$ within the theory; a new derivation of some hypergeometric identities of Ramanujan was a byproduct of his investigations. Subsequently, the elliptic function ${\rm dn}_2$ received an alternative treatment, briefly in [5] and more fully in [8]. 

\medbreak 

The signature six theory involves $F(\tfrac{1}{6}, \tfrac{5}{6}; 1 ; \bullet)$. In this case, Shen [10] brought an elliptic function into the theory, but by an entirely different mechanism, again offering a new perspective on certain hypergeometric identities of Ramanujan. In [6] we showed that the approaches adopted by Shen in signature three and signature four yield functions that are not elliptic when adopted in signature six. Our chief purpose in this brief paper is complementary: to note that, though non-elliptic, these functions still lead to the same hypergeometric identities of Ramanujan; we also connect these non-elliptic functions to the very differently sourced elliptic functions of Shen in this signature. 

\medbreak 

\section*{Signature six}

\medbreak 

Throughout, our focus will be on real-valued functions of a real variable. 

\medbreak 

Fix a choice $\kk \in (0, 1)$ of `modulus'. We define a function $f: \R \to \R$ by requiring that if $T \in \R$ then 
$$f(T) = \int_0^T F(\tfrac{1}{6}, \tfrac{5}{6}; \tfrac{1}{2} ; \kk^2 \sin^2 t)\, {\rm d} t.$$ 
As the continuous (indeed, analytic) integrand here is even and strictly positive, the function $f$ is odd and strictly increasing; $f$ is also surjective, as the integrand is periodic. Thus $f$ has an odd inverse function, which we denote by $\phi : \R \to \R$: so if $u \in \R$ then 
$$u = \int_0^{\phi (u)} F(\tfrac{1}{6}, \tfrac{5}{6}; \tfrac{1}{2} ; \kk^2 \sin^2 t)\, {\rm d} t.$$ 

\medbreak 

We now define an odd function $\rms : \R \to \R$ by following $\phi$ with the sine function, thus 
$$\rms = \sin \circ \, \phi;$$ 
we may also similarly define an even function $\rmc : \R \to \R$ by 
$$\rmc = \cos \circ \, \phi.$$ Of course, these two functions satisfy the `Pythagorean' identity 
$$\rmc^2 + \rms^2 = 1$$ 
along with 
$$\rms(0) = 0 \; \; {\rm and} \; \; \rmc(0) = 1.$$ 
The functions $\rms$ and $\rmc$ are counterparts to the Jacobian functions ${\rm sn}$ and ${\rm cn}$ familiar from the classical theory of elliptic functions. 
However, neither $\rms$ nor $\rmc$ admits an elliptic extension: see [6] for this and more. 
\medbreak 

As an analogue of the classical complete elliptic integral, we introduce 
$$K = f(\tfrac{1}{2} \pi) = \int_0^{\frac{1}{2} \pi}  F(\tfrac{1}{6}, \tfrac{5}{6}; \tfrac{1}{2} ; \kk^2 \sin^2 t)\, {\rm d} t\, ,$$
this being a function of the modulus $\kk$ (rather, of its square). Thus 
$$\phi (K) = \tfrac{1}{2} \pi$$ 
so that  
$$\rms (K) = 1 \; \; {\rm and} \; \; \rmc(K) = 0.$$ 

\medbreak

Not surprisingly, the functions $\rms$ and $\rmc$ are periodic, with $4 K$ as period. As a step towards the establishment of this fact, note first the identity 
$$\int_0^{\pi}  F(\tfrac{1}{6}, \tfrac{5}{6}; \tfrac{1}{2} ; \kk^2 \sin^2 t)\, {\rm d} t = f(\pi) = 2 K.$$ 
In fact, the substitution $t = \tau + \pi$ gives 
$$\int_{\frac{1}{2} \pi}^{\pi}  F(\tfrac{1}{6}, \tfrac{5}{6}; \tfrac{1}{2} ; \kk^2 \sin^2 t)\, {\rm d} t = \int_{- \frac{1}{2} \pi}^0  F(\tfrac{1}{6}, \tfrac{5}{6}; \tfrac{1}{2} ; \kk^2 \sin^2 \tau)\, {\rm d} \tau = K$$
since $\sin (\tau + \pi) = - \sin \tau$ and the integrand is even. 
\medbreak 

\begin{theorem} \label{2 K}
If $u \in \R$ then $\rms(u + 2 K) = - \rms (u)$ and $\rmc(u + 2K) = - \rmc(u).$
\end{theorem} 

\begin{proof} 
Let $T = \phi (u)$. Splitting the integral, 
$$f(\pi  + T) = \int_0^{\pi}  F(\tfrac{1}{6}, \tfrac{5}{6}; \tfrac{1}{2} ; \kk^2 \sin^2 t)\, {\rm d} t + \int_{\pi}^{\pi + T}  F(\tfrac{1}{6}, \tfrac{5}{6}; \tfrac{1}{2} ; \kk^2 \sin^2 t)\, {\rm d} t.$$ 
On the right, the first integral is $2 K$ (as seen above) while the second reduces to $f(T)$ after a $\pi$ shift in the variable. Thus 
$$f(\pi + T) = 2 K + f(T) = 2 K + u$$
and so 
$$\phi (2 K + u) = \pi + T= \pi + \phi(u).$$
The first identity of the Theorem follows upon applying the sine:  
$$\rms (2 K + u) = \sin \circ \, \phi (2 K + u) = \sin (\pi + \phi (u)) = - \sin ( \phi (u) ) = - \rms (u);$$ 
the second follows upon application of the cosine. 
\end{proof} 

\medbreak 

\begin{theorem} \label{4K}
The functions $\rms$ and $\rmc$ have period $4 K$. 
\end{theorem} 

\begin{proof} 
Let $g$ stand for $\rms$ or $\rmc$. If $u \in \R$ then 
$$g (u + 4 K) = g (u + 2 K + 2 K) = - g (u + 2 K) = - ( -g (u)) = g (u).$$ 
\end{proof} 

\medbreak  

The `complete integral' $K = f(\tfrac{1}{2} \pi)$ itself has a familiar value in hypergeometric terms. 

\medbreak 

\begin{theorem} \label{K}
$K = \tfrac{1}{2} \pi \, F(\tfrac{1}{6}, \tfrac{5}{6}; 1 ; \kk^2).$
\end{theorem} 

\begin{proof} 
Expand the hypergeometric integrand and integrate the resulting series term-by-term, using the standard integral formula 
$$\int_0^{\frac{1}{2} \pi} \sin^{2 n} t \, {\rm d} t = \tfrac{1}{2} \pi \, \tfrac{(2 n)!}{(2^n n!)^2}.$$ 
\end{proof} 

\medbreak 

Up until now, the hypergeometric parameters need not have been $1/6$ and $5/6$; henceforth, we shall require the parameters to have these precise values. 

\medbreak 

An alternative explicit integral formula for $K$ stems from the hypergeometric evaluation 
$$F(\tfrac{1}{6}, \tfrac{5}{6}; \tfrac{1}{2} ; \sin^2 \psi) = \frac{\cos \frac{2}{3} \psi}{\cos \psi}$$ 
for which we refer to page 101 in Volume 1 of the Bateman Manuscript Project [2]. 

\medbreak 

For the following, let $\alpha \in (0, \frac{1}{2} \pi)$ be the `modular' angle defined by 
$$\kk = \sin \alpha.$$ 

\medbreak 

\begin{theorem} \label{integral}
$$K = \sqrt{2} \, \int_0^{\alpha} \frac{\cos \frac{2}{3} \psi}{\sqrt{\cos 2 \psi - \cos 2 \alpha}} \, {\rm d} \psi.$$
\end{theorem} 

\begin{proof} 
In the definition 
$$K = \int_0^{\frac{1}{2} \pi}  F(\tfrac{1}{6}, \tfrac{5}{6}; \tfrac{1}{2} ; \kk^2 \sin^2 \varphi)\, {\rm d} \varphi$$ 
change the integration variable from $\varphi \in [0, \frac{1}{2} \pi]$ to $\psi \in [0, \alpha]$ where  
$$\sin \psi = \kk \, \sin \varphi.$$ 
The hypergeometric integrand evaluates as 
$$F(\tfrac{1}{6}, \tfrac{5}{6}; \tfrac{1}{2} ; \kk^2 \sin^2 \varphi) = F(\tfrac{1}{6}, \tfrac{5}{6}; \tfrac{1}{2} ; \sin^2 \psi) = \frac{\cos \frac{2}{3} \psi}{\cos \psi}$$ 
alongside which the change of variable introduces the factor 
$$\frac{{\rm d} \varphi}{{\rm d} \psi} = \frac{\cos \psi}{\kk \cos \varphi} = \frac{\cos \psi}{\sqrt{\kk^2 - \sin^2 \psi}}$$
in which trigonometric duplication gives 
$$\kk^2 - \sin^2 \psi = \sin^2 \alpha - \sin^2 \psi = \tfrac{1}{2} (\cos 2 \psi - \cos 2 \alpha).$$ 
\end{proof} 

\medbreak 

We shall now transform this integral expression for $K$ until it asssumes a decidedly elliptic appearance. 

\medbreak 

\begin{theorem} \label{Kell}
$$K = \sqrt{\frac{3}{2}} \int_{\cos \frac{2}{3} (\pi + \alpha)}^{\cos \frac{2}{3} (\pi - \alpha)} \frac{{\rm d} x}{\sqrt{4x^3 - 3 x - (1 - 2 \kk^2)}} \,.$$
\end{theorem} 

\begin{proof} 
First of all, as the cosine function is even, 
$$\sqrt{\frac{2}{3}} K = \frac{1}{\sqrt3}  \int_{- \alpha}^{\alpha} \frac{\cos \frac{2}{3} \psi}{\sqrt{\cos 2 \psi - \cos 2 \alpha}} \, {\rm d} \psi = \frac{2}{3} \sin \frac{2}{3} \pi \int_{- \alpha}^{\alpha} \frac{\cos \frac{2}{3} \psi}{\sqrt{\cos 2 \psi - \cos 2 \alpha}} \, {\rm d} \psi.$$
Next, the addition formula for the (odd) sine function yields 
$$\sqrt{\frac{2}{3}} K = \frac{2}{3} \int_{- \alpha}^{\alpha} \frac{\sin \frac{2}{3} (\pi + \psi)}{\sqrt{\cos 2 \psi - \cos 2 \alpha}} \, {\rm d} \psi$$
whence the substitution $\theta = \pi + \psi$ leads to 
$$\sqrt{\frac{2}{3}} K = \frac{2}{3} \int_{\pi - \alpha}^{\pi + \alpha} \frac{\sin \frac{2}{3} \theta}{\sqrt{\cos 2 \theta - \cos 2 \alpha}} \, {\rm d} \theta.$$
Finally, the substitution $x = \cos \frac{2}{3} \theta$ produces the integral formula 
$$\sqrt{\frac{2}{3}} K = \int_{\cos \frac{2}{3} (\pi + \alpha)}^{\cos \frac{2}{3} (\pi - \alpha)} \frac{{\rm d} x}{\sqrt{4x^3 - 3 x - \cos 2 \alpha}}$$
wherein $\cos 2 \alpha = 1 - 2 \sin^2 \alpha = 1 - 2 \kk^2$. 
\end{proof} 

\medbreak 

The idea of this proof comes straight from Theorem 3.2 in [10], though we have varied its place in the scheme of things. 

\medbreak 

As announced, this integral formula is elliptic in nature. Indeed, let $\wp$ be the Weierstrass function with invariants 
$$g_2 = 3 \; \; {\rm and} \; \; g_3 = 1 - 2 \kk^2$$ 
so that 
$$(\wp ')^2 = 4 \wp^3 - 3 \wp - (1 - 2 \kk^2) = 4 \wp^3 - 3 \wp - \cos 2 \alpha.$$ 
\medbreak 
\noindent
The invariants being real and the discriminant 
$$\Delta = g_2^3 - 27 g_3^2 = 108 \kk^2 (1 - \kk^2)$$ 
being positive, $\wp$ has a rectangular period lattice with a positive fundamental half-period $\omega$. 

\medbreak 

\begin{theorem} \label{omega}
$K = \sqrt{\frac{3}{2}} \, \omega.$
\end{theorem} 

\begin{proof} 
The roots $e_1 > e_2 > e_3$ of the cubic $4 x^3 - 3 x - (1 - 2 \kk^2)$ are precisely 
$$e_1 = \cos \tfrac{2}{3} \alpha, \; e_2 = \cos \tfrac{2}{3} (\alpha - \pi), \; e_3 = \cos \tfrac{2}{3} (\alpha + \pi).$$
Simply substitute these values into the formula 
$$\omega = \int_{e_3}^{e_2} \frac{{\rm d} x}{\sqrt{4x^3 - 3 x - (1 - 2 \kk^2)}}$$
for which we refer to Section 51 of the classic text [3] by Greenhill. 
\end{proof} 

\medbreak 

Having identified $K$ in classical elliptic terms, we may of course also identify $K$ in terms of the `classical' hypergeometric function $F(\tfrac{1}{2}, \tfrac{1}{2}; 1 ; \bullet)$. 

\medbreak 

For this, it is convenient to rescale the acute modular angle $\alpha$ and define $\beta \in (0, \frac{1}{3} \pi)$ by 
$$3 \beta = 2 \alpha.$$ 

\medbreak 

\begin{theorem} \label{Kclass}
$$K = (1 - k^2 + k^4)^{1/4} \, \tfrac{1}{2} \, \pi \, F(\tfrac{1}{2}, \tfrac{1}{2}; 1 ; k^2)$$
where the (classical) modulus $k$ is given by 
$$k^2 = 2 \, \frac{\sin \beta}{\sin \beta + \sqrt3 \cos \beta}.$$
\end{theorem} 

\begin{proof} 
Further reference to Section 51 of [3] turns up the following formula for the positive fundamental half-period $\omega$ of $\wp$: 
$$\omega = \sqrt{\frac{1}{e_1 - e_3}} \, K \Big(\frac{e_2 - e_3}{e_1 - e_3}\Big)$$
where $K(k^2)$ is the complete elliptic integral with $k$ as classical modulus, which evaluates hypergeometrically as 
$$K(k^2) = \tfrac{1}{2} \pi \, F(\tfrac{1}{2}, \tfrac{1}{2}; 1 ; k^2).$$ 
The formulae 
$$g_2 = 2 \, ( e_1^2 + e_2^2 + e_3^3) = - 4 \, (e_2 e_3 + e_3 e_1 + e_1 e_2)$$
relating the invariant $g_2$ to the midpoint values are standard: see Section 53 of [3] for example. With 
$$k^2 = \frac{e_2 - e_3}{e_1 - e_3}$$ 
it follows that 
$$(e_1 - e_3)^2 (1 - k^2 + k^4) = ( e_1^2 + e_2^2 + e_3^3) - (e_2 e_3 + e_3 e_1 + e_1 e_2) = \tfrac{3}{4} g_2 = \tfrac{9}{4}$$
here, whence 
$$ \sqrt{\frac{3}{2}} \, \sqrt{\frac{1}{e_1 - e_3}} = (1 - k^2 + k^4)^{1/4}.$$ 
\medbreak
\noindent
In the proof of Theorem \ref{omega} we recorded the midpoint values 
$$e_1 = \cos \beta, \; e_2 = \cos (\beta - \tfrac{2}{3} \pi), \; e_3 = \cos (\beta + \tfrac{2}{3} \pi).$$ 
From these formulae, it follows that $e_1 - e_3 = \sqrt3 \sin(\beta + \tfrac{1}{3} \pi)$ and $e_2 - e_3 = \sqrt3 \sin \beta$, whence 
$$k^2 = \frac{\sin \beta}{\sin(\beta + \tfrac{1}{3} \pi)} = 2 \, \frac{\sin \beta}{\sin \beta + \sqrt3 \cos \beta}.$$
Assemble the pieces to conclude the proof. 
\end{proof} 

\medbreak 

We are now in possession of a hypergeometric identity that relates signature six to the classical signature. Direct comparison of Theorem \ref{K} and Theorem \ref{Kclass} reveals that 
$$F(\tfrac{1}{6}, \tfrac{5}{6}; 1 ; \kk^2) = (1 - k^2 + k^4)^{1/4} \, F(\tfrac{1}{2}, \tfrac{1}{2}; 1 ; k^2)$$
where 
$$\kk = \sin \alpha \; \; {\rm and} \; \; k^2 = 2 \, \frac{\sin \beta}{\sin \beta + \sqrt3 \cos \beta}$$
with 
$$2 \alpha = 3 \beta.$$ 
We can recast this relationship into more familiar form, as the following hypergeometric identity. 

\medbreak 

\begin{theorem} \label{hyper}
For the increasing bijection $(0, 1) \to (0, 1): x \mapsto \xi$ given by 
$$4 \xi (1 - \xi) = \frac{27}{4} \, \frac{x^2 (1 - x)^2}{(1 - x + x^2)^3}$$ 
there holds the identity  
$$F(\tfrac{1}{6}, \tfrac{5}{6}; 1 ; \xi) = (1 - x + x^2)^{1/4} \, F(\tfrac{1}{2}, \tfrac{1}{2}; 1 ; x).$$
\end{theorem} 

\begin{proof} 
Two increasing bijections are defined by 
$$(0, \tfrac{1}{2} \pi) \to (0, 1) : \alpha \mapsto \sin^2 \alpha = \kk^2 = : \xi$$ 
and 
$$(0, \tfrac{1}{3} \pi) \to (0, 1) : \beta \mapsto 2 \, \frac{\sin \beta}{\sin \beta + \sqrt3 \cos \beta} = k^2 = : x.$$ 
\medbreak
\noindent
In light of the identity revealed just prior to this Theorem, we need only check that when the rescaling $2 \alpha = 3 \beta$ is used to coordinate these bijections, there results the bijection of the Theorem. Thus, let $0 < \beta < \tfrac{1}{3} \pi$: from 
$$\frac{1}{x} = \frac{1}{2} (1 + \sqrt3 \, \cot \beta)$$ 
it follows that 
$$\frac{1}{x^2} - \frac{1}{x} + 1 = \frac{1}{4} ( 3 + 3 \cot^2 \beta) = \frac{3}{4} \csc^2 \beta$$ 
and therefore that 
$$\frac{3}{4} \frac{x^2}{1 - x + x^2} = \sin^2 \beta = : S$$
say; now 
$$3 - 4S = 3 \, \frac{1 - x}{1 - x + x^2}$$ 
so that by trigonometric triplication
$$\frac{27}{4} \, \frac{x^2 (1 - x)^2}{(1 - x + x^2)^3} = S (3 - 4S)^2 = \sin^2 3 \beta$$
\medbreak 
\noindent
and the coordinating rescale gives 
$$\sin^2 3 \beta = \sin^2 2 \alpha$$ 
where by trigonometric duplication 
$$\sin^2 2 \alpha = 4 \sin^2 \alpha \cos^2 \alpha = 4 \xi (1 - \xi).$$ 
\end{proof} 

\medbreak 

We spoke of this hypergeometric identity as being familiar. In fact, it appears as equation (1.5) in [10]; the proof in Section 4 of that paper rests in part on the theory of theta functions. 

\medbreak 

Additional hypergeometric identities follow from this one. Note first that the bijection in Theorem \ref{hyper} is invariant under the simultaneous replacements $x \mapsto 1 - x$ and $\xi \mapsto 1 - \xi$. Now, let $0 < p < 1$: on the one hand, if we put 
$$x = \frac{p(2 + p)}{1 + 2 p} \; \; \; {\rm and} \; \; \; \xi = \frac{27}{4} \frac{p^2 (1 + p)^2}{(1 + p + p^2)^3}$$ 
then 
$$1 - x + x^2 = \Big(\frac{1 + p + p^2}{1 + 2 p}\Big)^2$$ 
and Theorem \ref{hyper} exactly reproduces Theorem 11.1 of [1]; on the other hand, if we make the aforementioned simultaneous replacements 
$$x = \frac{1 - p^2}{1 + 2 p} \; \; \; {\rm and} \; \; \; \xi = \frac{1}{4} \frac{(1 - p)^2 (1 + 2 p)^2 (2 + p)^2}{(1 + p + p^2)^3}$$ 
then $1 - x + x^2$ remains the same and Theorem \ref{hyper} exactly reproduces Corollary 11.2 of [1]. Of course, the proofs of these identities in [1] are rather different, being based instead on judicious manipulations of other hypergeometric identities recorded in the Bateman Manuscript Project [2] (specifically: (42) on page 114 with $a = 1/12$; (32) on page 113 with $a = b = 1/2$; (2) on page 111 with $a = 1/12$ and $b = 5/12$). 

\medbreak

\bigbreak

\begin{center} 
{\small R}{\footnotesize EFERENCES}
\end{center} 
\medbreak 

[1] B.C. Berndt, S. Bhargava, and F.G. Garvan, {\it Ramanujan's theories of elliptic functions to alternative bases}, Transactions of the American Mathematical Society {\bf 347} (1995) 4163-4244. 

\medbreak 

[2] A. Erdelyi (director), {\it Higher Transcendental Functions}, Volume 1, McGraw-Hill (1953). 

\medbreak 

[3] A.G. Greenhill, {\it The Applications of Elliptic Functions}, Macmillan and Company (1892); Dover Publications (1959). 

\medbreak 

[4] P.L. Robinson, {\it Elliptic functions from $F(\frac{1}{3}, \frac{2}{3} ; \frac{1}{2} ; \bullet)$}, arXiv 1907.09938 (2019). 

\medbreak 

[5] P.L. Robinson, {\it Elliptic functions from $F(\tfrac{1}{4}, \tfrac{3}{4}; \tfrac{1}{2} ; \bullet)$}, arXiv 1908.01687 (2019). 

\medbreak 

[6] P.L. Robinson, {\it Nonelliptic functions from $F(\tfrac{1}{6}, \tfrac{5}{6}; \tfrac{1}{2} ; \bullet)$}, arXiv 2004.06529 (2020). 

\medbreak 

[7] P.L. Robinson, {\it The elliptic function ${\rm dn}_3$ of Shen}, arXiv 2008.13572 (2020). 

\medbreak 

[8] P.L. Robinson, {\it The elliptic function ${\rm dn}_2$ of Shen}, arXiv 2009.04910 (2020). 

\medbreak 

[9] Li-Chien Shen, {\it On the theory of elliptic functions based on $_2F_1(\frac{1}{3}, \frac{2}{3} ; \frac{1}{2} ; z)$}, Transactions of the American Mathematical Society {\bf 357}  (2004) 2043-2058. 

\medbreak 

[10] Li-Chien Shen, {\it A note on Ramanujan's identities involving the hypergeometric function $F(\tfrac{1}{6}, \tfrac{5}{6}; 1 ; z)$}, Ramanujan Journal {\bf 30} (2013) 211-222. 

\medbreak 

[11] Li-Chien Shen, {\it On a theory of elliptic functions based on the incomplete integral of the hypergeometric function $_2 F_1 (\frac{1}{4}, \frac{3}{4} ; \frac{1}{2} ; z)$}, Ramanujan Journal {\bf 34} (2014) 209-225. 

\medbreak

\end{document}